\documentclass[a4paper,11pt]{article}
\usepackage{amssymb, amsmath, amsfonts}
\usepackage{epic}
\usepackage{pstricks}
\usepackage{pst-node}
\usepackage{pstricks-add}
\usepackage[numbers,sort&compress]{natbib}
\usepackage{multicol,caption}

\newcommand{\comment}[1]{}
\newtheorem{thm}{Theorem}[section]

\newtheorem{cor}[thm]{Corollary}

\newtheorem{lem}[thm]{Lemma}

\newtheorem{prop}[thm]{Proposition}
\newtheorem{rem}[thm]{Remark}

\newtheorem{probl}{Problem}
\newenvironment{proof}{\noindent {\bf Proof.}}{$\Box$\\}
\newcommand{\fd}{{\rm fd}}

\newcommand{\rep}{{\rm rep}}
\newcommand{\reg}{\alpha_{\rm reg}}
\newcommand{\kreg}{\alpha_{k-{\rm reg}}}
\newcommand{\jreg}{\alpha_{j-{\rm reg}}}
\newcommand{\tworeg}{\alpha_{2-{\rm reg}}}
\newcommand{\onereg}{\alpha_{1-{\rm reg}}}

\begin{document}

\begin{center}
{\Large \bf Regular independent sets}\\[5ex]

\begin{multicols}{2}

Yair Caro\\[1ex]
{\small Dept. of Mathematics and Physics\\
University of Haifa-Oranim\\
Tivon 36006, Israel\\
yacaro@kvgeva.org.il}

\columnbreak

Adriana Hansberg\\[1ex]
{\small Instituto de Matem\'aticas\\
UNAM Juriquilla\\
76230 Quer\'etaro, Mexico\\
ahansberg@im.unam.mx}\\[2ex]

\end{multicols}

Ryan Pepper\\[1ex]
{\small University of Houston-Downtown\\
Houston, Texas 77002\\
pepperr@uhd.edu}
\end{center}

\begin{abstract}
The regular independence number, introduced by Albertson and Boutin in 1990, is the size of a largest set of independent vertices with the same degree.  Lower bounds were proven for this invariant, in terms of the order, for trees and planar graphs.  In this article, we generalize and extend these results to find lower bounds for the regular $k$-independence number for trees, forests, planar graphs, $k$-trees and $k$-degenerate graphs.\\[1ex]
\end{abstract}

\noindent
\section{Introduction and benchmark bounds}

Albertson and Boutin \cite{AlBo} introduced the parameter $\reg(G)$ as the maximum cardinality of an independent set in a graph $G$ in which all vertices have equal degree in $G$. An independent set whose vertices all have equal degree in $G$ is called a \emph{regular independent set}. 
A \textit{$k$-independent set} is a set of vertices whose induced subgraph has maximum degree at most $k$. Let us define the \textit{regular $k$-independence number}, denoted $\alpha_{k-reg}(G)$, as the maximum cardinality of a $k$-independent set of vertices which have the same degree in $G$.  More particularly, we denote by $\alpha_{k,j}(G)$ the maximum cardinality of a $k$-independent set in the subgraph induced by the vertices of degree $j$ in $G$.  Thus, $\alpha_{k-reg}(G) = \max \{\alpha_{k,j}(G) \,: \delta \leq j \leq \Delta\}$, where $\delta$ is the minimum degree and $\Delta$ is the maximum degree. When $k=0$,  $\alpha_{0-reg}(G) = \alpha_{reg}(G)$ and, for regular graphs, $\reg(G) = \alpha(G)$ and $\kreg(G) = \alpha_k(G)$. Albertson and Boutin \cite{AlBo} proved the following:
\begin{enumerate}
\item[(1)] If $G$ is a planar graph on $n$ vertices, then $\reg(G) \ge \frac{2}{65}n$.
\item[(2)] If $G$ is a maximal planar graph on $n$ vertices, then $\reg(G) \ge \frac{3}{61}n$.
\item[(3)] If $T$ is a tree on $n$ vertices, then $\reg(T) \ge \frac{n+2}{4}$ and this is sharp.
\item[(4)] If $G$ is a connected graph on $n$ vertices and maximum degree $\Delta$, then $\reg(G) \ge \frac{n}{{{\Delta+1}\choose 2}}$ and this is sharp.
\end{enumerate}

They left open the problems whether the results in (1) and (2) are best possible and constructed an example of a maximal planar graph $G$ for which $\reg(G) = \frac{n}{16}$. Our intention is to extend the issue raised by Albertson and Boutin in two directions. We extend the question to broader families of graphs, including forests, $k$-trees and $k$-degenerate graphs, while on the way improving item~(1) for minimum degree $\delta = 2,3,4,5$ and item~(2) for minimum degree $\delta = 4, 5$. These results are summarized in Table \ref{table-og} in what follows. We also extend the problem from regular independent sets to regular $k$-independent sets, in view to some recent results about the $k$-independence number \cite{CaHa, HaPe}. We mention in passing that related papers were written (that came from distinct inspiration), see e.g. \cite{AlBe, AlMuTh, BieWil, BoDuWo, CaWe}, from which some lower bounds on the regular $k$-independence number can be obtained, but which are inferior to the lower bound obtained by the current approach.


The parameter $\reg(G)$ is closely related to a newly introduced parameter \cite{CHH}, the fair domination number $\fd(G)$. A \emph{fair dominating set} is a set $S \subseteq V(G)$ such that all vertices $v \in V(G) \setminus S$ have exactly the same non-zero number of neighbors in $S$. The \emph{fair domination number} $\fd(G)$ is the cardinality of a minimum fair dominating set of $G$. When $\delta(G) \ge 1$ and $R$ is a maximum regular independent set of $G$, then $V(G) \setminus R$ is a fair dominating set of $G$ and hence $\fd(G) \le n - \reg(G)$. Hence, any lower bound on $\reg(G)$ yields an upper bound on $\fd(G)$ (and any lower bound on $\fd(G)$ yields an upper bound on $\reg(G)$).

The lower bound given in the proposition presented below, together with the above mentioned results of Albertson and Boutin, serves as a benchmark to our work. A few more definitons are necessary before proceeding. The \emph{repetition number} of $G$, denoted $\rep(G)$, was introduced in \cite{CaWe} and is defined as the maximum number of vertices with equal degree, while $\chi_k(G)$ is the \emph{$k$-chromatic number of $G$}, i.e. the minimum number of colors needed to color the vertices of the graph $G$ such that the graphs induced by the vertices of each color class have maximum degree at most $k$. Note that $\chi_0(G)$ is the well-known chromatic number $\chi(G)$.

\begin{prop}\label{prop1} 
Let $G$ be a graph with average degree $d$ and minimum degree $\delta$. Then
\[
\kreg(G) \ge \frac{n}{(2 d - 2\delta +1) \chi_k(G)}.
\]
\end{prop}

\begin{proof}
Let $G_i$ be the subgraph of $G$ induced by all vertices of degree $i$, $\delta \le i \le \Delta$. Then we have
\[\kreg(G)  = \max \{\alpha_{k,j}(G) \,: \delta \leq j \leq \Delta\}
          \ge \max \left\{\frac{|V(G_j)|}{\chi_k(G_j)} \,:\; \delta \leq j \leq \Delta \right\},\]
since, for every $j$, $\alpha_{k,j}(G) \ge \frac{|V(G_j)|}{\chi_k(G_j)}$ holds. Hence, with
 \[\max \left\{\frac{|V(G_j)|}{\chi_k(G_j)} \,:\, \delta \leq j \leq \Delta \right\}
 \ge \max \left\{\frac{\rep(G)}{\chi_k(G_j)} \,:\, |V(G_j)| = \rep(G),\; \delta \leq j \leq \Delta\right\} \ge\frac{\rep(G)}{\chi_k(G)}
 \]
and the bound $\rep(G) \ge \frac{n}{2 d - 2\delta +1}$ proved in \cite{CaWe}, the result follows.
\end{proof}

The comparison of the benchmark results obtained in Proposition \ref{prop1} and the outcome of our more intricate approach are presented in Table \ref{table-og} and Table \ref{table-p-op-2} later on.

Definitions for terms not defined in the introduction but necessary for the paper appear in the sections that they are needed. This paper is organized as follows. After this introduction section, in Section 2 we deal with trees and forests, where we generalize and extend the results of Albertson and Boutin in \cite{AlBo} to $\kreg(G)$. In Section 3, we present a lower bound on $\reg(G)$ for $k$-trees (observe that no planar graph is a $4$-tree and planar graphs which are $3$-trees are called Apollonian networks) and give analogous results for $k$-degenerate graphs (observe that all planar graphs are at most $5$-degenerate graphs) and some specific results about planar graphs, refining and elaborating hereby the method used in \cite{AlBo} and generalizing and improving the bounds on $\reg(G)$ given there, too. In Section 4 we give lower bounds on $\tworeg(G)$ for planar and outerplanar graphs by incorporating ideas from defective colorings \cite{CoCoWo, CoGoJe} and a partition theorem due to Lov\'asz \cite{Lov}. In Section 5, we analyze complexity issues of regular $k$-indenpendence and we finish with a collection of open problems on regular $k$-independent sets in Section 6.

\section{Trees and forests}

In this section, we generalize and extend the result that $\alpha_{reg}(T)\geq \frac{n+2}{4}$ for any tree $T$, obtained by Albertson and Boutin in \cite{AlBo}, to regular $k$-independence number in both trees and forests. Here and throughout, we will use $n_i(G)$ to denote the number of vertices of degree $i$ in $G$.  When the context is clear, $n_i(G)$ will be abbreviated to $n_i$. As usual, a \emph{leaf} of a tree $T$ is a vertex of degree $1$ in $T$ and its neighbor is called its \emph{support vertex}.

\begin{thm}\label{thm-trees} 
For every tree $T$ on $n \geq 2$ vertices, \\[-4ex]
\begin{enumerate}
\item[(i)] $\reg(T) \ge \frac{n+2}{4}$ (Albertson and Boutin \cite{AlBo}),
\item[(ii)] $\onereg(T) \ge \frac{2(n+2)}{7}$, and
\item[(iii)] $\kreg(T) \ge \frac{n+2}{3}$ for $k \ge 2$.\\[-4ex]
\end{enumerate}
Moreover, all bounds are sharp.
\end{thm}

\begin{proof}
Let $T$ be a tree on $n \geq 2$ vertices. Item (i) is proved, and the sharpness is demonstrated, in \cite{AlBo}. Is is only mentioned here for completeness.

To prove (ii), we first assume that $n_1 \geq \frac{2(n+2)}{7}$.  Since the subgraph induced by the vertices of degree one is a $1$-independent set, $\alpha_{1,1} = n_1$ and we are done.  Next, we assume that $n_1 < \frac{2(n+2)}{7}$.  Denote by $N_3$ the set of vertices in $T$ of degree at least $3$. Then $n_1 \geq N_3 +2$ and $n_2=n-n_1-N_3$ and so $n_2 \geq n - 2n_1 +2$. Finally, since the vertices of degree two in $T$ induce a collection of paths, and the regular $1$-independence number of a path is at least $\frac{2}{3}$ of its order, $\alpha_{1,2}(T) \geq \frac{2}{3}n_2$. Thus, we complete the proof as follows:
\[
\onereg(T) \geq \alpha_{1,2}(T) \geq \frac{2}{3}n_2 \geq \frac{2}{3}(n - 2n_1 +2) > \frac{2}{3} \left(n - \frac{4(n+2)}{7} +2\right) = \frac{2(n+2)}{7}.
\]
To see that this bound is sharp, let $p$ be a positive integer and $n=7p+5$. Start with a path $P$ with $\frac{3n+6}{7}$ vertices and two copies, $T_1$ and $T_2$, of a tree with $\frac{2n-3}{7}$ vertices which has a unique degree two vertex and maximum degree at most three (these trees exist since we can start with a path on three vertices and continually add two leaves to an endpoint). Now join one endpoint of $P$ to the unique degree-two vertex of one $T_1$ and the other endpoint of $P$ to the unique degree two vertex of $T_2$. Now, for this new tree we constructed, $n_1=\frac{2n+4}{7}$, $n_2=\frac{3n+6}{7}$, $n_3=\frac{2n-10}{7}$, and $\onereg=\alpha_{1,1}=\alpha_{1,2}=\frac{2n+4}{7}$ -- which shows the bound is sharp (see Figure \ref{fig-trees-ii}).

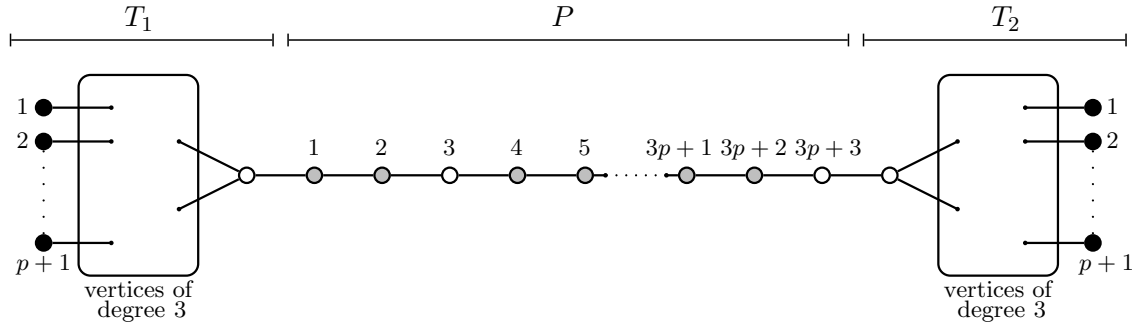
\begin{figure}[h]
\begin{center}
\psset{unit=0.9cm, linewidth=0.03cm}
   \begin{pspicture}(-2.5,0)(13,5.3)
   \cnode[fillstyle=solid,fillcolor=lightgray](1.5,2){0.13}{a1}\put(1.4,2.3){\footnotesize $1$}
      \cnode[fillstyle=solid,fillcolor=lightgray](2.5,2){0.13}{a2}\put(2.4,2.3){\footnotesize $2$}
      \cnode(3.5,2){0.13}{a3}\put(3.4,2.3){\footnotesize $3$}
      \cnode[fillstyle=solid,fillcolor=lightgray](4.5,2){0.13}{a4}\put(4.4,2.3){\footnotesize $4$}
      \cnode[fillstyle=solid,fillcolor=lightgray](5.5,2){0.13}{a5}\put(5.4,2.3){\footnotesize $5$}
      \cnode[fillstyle=solid,fillcolor=lightgray](7,2){0.13}{a7}\put(6.4,2.3){\footnotesize $3p+1$}
      \cnode[fillstyle=solid,fillcolor=lightgray](8,2){0.13}{a8}\put(7.5,2.3){\footnotesize $3p+2$}
      \cnode(9,2){0.13}{a9}\put(8.6,2.3){\footnotesize $3p+3$}
      \cnode(5.8,2){0}{x1}
      \cnode(6.7,2){0}{x2}
      
       \ncline{a1}{a2}
      \ncline{a2}{a3}
      \ncline{a3}{a4}
      \ncline{a4}{a5}
      \ncline{a5}{x1}
      \ncline{x2}{a7}
      \ncline[linestyle=dotted]{x1}{x2}
      \ncline{a6}{a7}
      \ncline{a7}{a8}
      \ncline{a8}{a9}
      
      \cnode(0.5,2){0.13}{v1}
      \cnode(-0.5,2.5){0}{y1}
      \cnode(-0.5,1.5){0}{y2}
      \psframe[framearc=0.2](-2,0.5)(-0.2,3.5)\put(-1.9,0.2){\footnotesize vertices of}\put(-1.7,-0.1){\footnotesize degree $3$}
      \cnode*(-2.5,3){0.13}{j1}\put(-2.9,2.9){\footnotesize $1$}
      \cnode*(-2.5,2.5){0.13}{j2}\put(-2.9,2.4){\footnotesize $2$}
      \cnode*(-2.5,1){0.13}{j3}\put(-2.9,0.6){\footnotesize $p+1$}
      \cnode(-1.5,3){0}{k1}
      \cnode(-1.5,2.5){0}{k2}
      \cnode(-1.5,1){0}{k3}
      \ncline{v1}{a1}
      \ncline{v1}{y1}
      \ncline{v1}{y2}
      \ncline{k1}{j1}
      \ncline{k2}{j2}
      \ncline{k3}{j3}
      \ncline[linestyle=dotted,dotsep=5pt]{j2}{j3}
   
      \cnode(10,2){0.13}{v2}
      \cnode(11,2.5){0}{z1}
      \cnode(11,1.5){0}{z2}
      \psframe[framearc=0.2](10.7,0.5)(12.5,3.5)\put(10.8,0.2){\footnotesize vertices of}\put(11,-0.1){\footnotesize degree $3$}
      \cnode*(13,3){0.13}{l1}\put(13.2,2.9){\footnotesize $1$}
      \cnode*(13,2.5){0.13}{l2}\put(13.2,2.4){\footnotesize $2$}
      \cnode*(13,1){0.13}{l3}\put(12.8,0.6){\footnotesize $p+1$}
      \cnode(12,3){0}{m1}
      \cnode(12,2.5){0}{m2}
      \cnode(12,1){0}{m3}
      \ncline{v2}{a9}
      \ncline{v2}{z1}
      \ncline{v2}{z2}
      \ncline{l1}{m1}
      \ncline{l2}{m2}
      \ncline{l3}{m3}
      \ncline[linestyle=dotted,dotsep=5pt]{l2}{l3}
      
      \psline[linewidth=0.02cm]{|-|}(-3,4)(0.9,4)\put(-1.3,4.2){$T_1$}
      \psline[linewidth=0.02cm]{|-|}(1.1,4)(9.4,4)\put(5,4.2){$P$}
      \psline[linewidth=0.02cm]{|-|}(9.6,4)(13.5,4)\put(11.5,4.2){$T_2$}
      
   \end{pspicture}
\caption{ \label{fig-trees-ii} Trees with $\onereg(T) = \frac{2(n(T)+2)}{7}$. The gray and the black vertices each form a different  $\onereg(T)$-set.}
\end{center}
\end{figure}

To prove (iii), let $k \geq 2$ and assume first that $n_1 \geq \frac{n+2}{3}$.  Since the subgraph induced by the vertices of degree one is a $1$-independent set, and therefore also a $k$-independent set, $\alpha_{k-reg} \geq \alpha_{k,1} = n_1$, and so we are done in this case.  Next, we assume that $n_1 < \frac{n+2}{3}$.  Now, we have as before that $n_2 \geq n - 2n_1 +2$, since $n_1 \geq N_3 +2$ and $n_2=n-n_1-N_3$.  Finally, since the subgraph induced by the vertices of degree two is a $2$-independent set:
\[
\alpha_{k-reg}(T) \geq \alpha_{k,2}(T) = n_2 \geq n - 2n_1 + 2 \geq n - \frac{2n+4}{3} + 2 = \frac{n+2}{3}.
\]
To see that this bound is sharp, consider the following tree. Let $p$ be a non-negative integer. Starting with a path on $r=2p+4$ vertices labeled $1,2,\ldots,r$, attach a leaf (or pendant vertex) to each of the first $p$ vertices with even labels (the first graph in this family is the path on four vertices).  This family of trees has the following properties; $n_1=n_2=n_3+2=p+2$, $n=3p+4$, and $\alpha_{k-reg} = \alpha_{k,1} = \alpha_{k,2} = p+2 = \frac{n+2}{3}$ -- which shows the bound is sharp (see Figure \ref{fig-trees-iii}).
\end{proof}

\begin{figure}[h]
\begin{center}
\psset{unit=1cm, linewidth=0.03cm}
   \begin{pspicture}(-1,0.7)(11,3.5)
      \cnode*(0,2){0.13}{a1}\put(-0.1,2.3){\footnotesize $1$}
      \cnode(1,2){0.13}{a2}\put(0.9,2.3){\footnotesize $2$}
      \cnode[fillstyle=solid,fillcolor=lightgray](2,2){0.13}{a3}\put(1.9,2.3){\footnotesize $3$}
      \cnode(3,2){0.13}{a4}\put(2.9,2.3){\footnotesize $4$}
      \cnode[fillstyle=solid,fillcolor=lightgray](4,2){0.13}{a5}\put(3.9,2.3){\footnotesize $5$}
      \cnode(6,2){0.13}{a6}\put(5.8,2.3){\footnotesize $2p$}
      \cnode[fillstyle=solid,fillcolor=lightgray](7,2){0.13}{a7}
      \cnode[fillstyle=solid,fillcolor=lightgray](8,2){0.13}{a8}
      \cnode[fillstyle=solid,fillcolor=lightgray](9,2){0.13}{a9}
      \cnode*(10,2){0.13}{a10}
      \cnode*(1,1){0.13}{a11}
      \cnode*(3,1){0.13}{a12}
      \cnode*(6,1){0.13}{a13} 
      \cnode(4.3,2){0}{x1}
      \cnode(5.7,2){0}{x2}
      
      \ncline{a1}{a2}
      \ncline{a2}{a3}
      \ncline{a3}{a4}
      \ncline{a4}{a5}
      \ncline{a5}{x1}
      \ncline[linestyle=dotted]{x1}{x2}
      \ncline{x2}{a6}
      \ncline{a6}{a7}
      \ncline{a7}{a8}
      \ncline{a8}{a9}
      \ncline{a9}{a10}
      \ncline{a2}{a11}
      \ncline{a4}{a12}
      \ncline{a6}{a13}

\end{pspicture}
\caption{ \label{fig-trees-iii} Trees with $\kreg(T) = \frac{n(T)+2}{3}$ for $k \ge 2$. The gray and the black vertices each form a different  $\kreg(T)$-set.}
\end{center}
\end{figure}
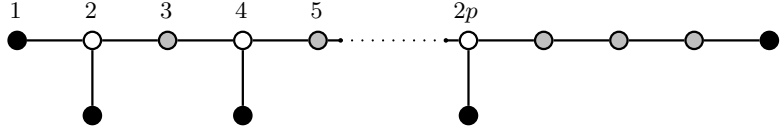

\newpage
\begin{thm}\label{thm-forests}
For every forest $F$, \\[-4ex]
\begin{enumerate}
\item[(i)] $\reg(F) \ge \frac{n+2}{5}$,
\item[(ii)] $\onereg(F) \ge \frac{2(n+2)}{9}$, and
\item[(iii)] $\kreg(F) \ge \frac{n+2}{4}$ for $k \ge 2$.\\[-4ex]
\end{enumerate}
Moreover, all bounds are sharp.
\end{thm}

\begin{proof} Let $F$ be a forest with $n$ vertices, $n_0$ isolated vertices, $t$ isolated edges (isolated complete graphs on two vertices) and $r$ trees with at least three vertices.  If $n\leq 2$, all three parts are trivially true, so we may assume throughout the proof that $n \geq 3$.

To prove (i), we first suppose $r=0$, which in turn implies that $n=n_0 + 2t$. Now, if $n_0 > \frac{n}{3}$, we are done since the degree-zero vertices form an independent set and $\frac{n}{3} \geq \frac{n+2}{5}$ when $n \geq 3$. So assume $n_0 \leq \frac{n}{3}$.  This implies $2t = n - n_0 \geq n - \frac{n}{3} = \frac{2n}{3}$ so that $t \geq \frac{n}{3}$.  Now, since taking one vertex of degree one from each of the $t$ isolated edges yields an independent set, $\alpha_{reg} \geq t \geq \frac{n}{3} \geq \frac{n+2}{5}$ when $n \geq 3$. This proves (i) when $r=0$, so we may assume that $r \geq 1$. 

Let $T_1, \ldots, T_r$ be the $r$ components of $F$ with at least three vertices and let $t_i$ denote the number of vertices of degree one in $T_i$ and let $N_3$ be the number of vertices of degree three or more in $F$. Then, 
\begin{equation}\label{r1}
n=n_0+n_1+n_2+N_3 = n_0 +2t + \sum_{i=1}^{r} t_i + n_2 + N_3.
\end{equation}
Note that in each component tree $T_i$, the number of vertices of degree at least three is at most $t_i -2$, with equality holding if and only if no vertices in that component have degree four or more -- equivalently, $n_3=N_3$. Hence,
\begin{equation}\label{r2}
N_3 \leq \sum_{i=1}^r (t_i - 2) = \sum_{i=1}^r t_i - 2r = n_1 - 2(t+r).
\end{equation}
Combining Equations \ref{r1} and \ref{r2},
\[
n=n_0+n_1+n_2+N_3 \leq n_0 + 2n_1 + n_2 - 2(t+r)
\]
so that,
\begin{equation}\label{r3}
n_2 \geq n - n_0 - 2n_1 +2(t+r).
\end{equation}

Since the degree-zero vertices are independent, if $n_0 > \frac{n+2}{5}$, we are done -- so assume $n_0 \leq \frac{n+2}{5}$.  Next, if $n_1 > \frac{n+2+5t}{5} = \frac{n+2-5t}{5}+2t$, then there are at least $\frac{n+2-5t}{5}+t = \frac{n+2}{5}$ independent vertices of degree one, implying $\alpha_{reg} \geq \frac{n+2}{5}$.  So we may also assume that $n_1 \leq \frac{n+2+5t}{5}$.  Under these assumptions, and using Equation \ref{r3}, we get;
\[
n_2 \geq n - \frac{n+2}{5} - \frac{2(n+2+5t)}{5} + 2(t+r) = \frac{2n}{5} - \frac{6}{5} + 2r.
\]
Clearly, at least half of the degree-two vertices are an independent set so that $\alpha_{reg} \geq \frac{n_2}{2}$.  Finally, because in this case, $r \geq 1$,
\[
\alpha_{reg} \geq \frac{n_2}{2} \geq \frac{n}{5} - \frac{3}{5} + r \geq \frac{n}{5} - \frac{3}{5} + 1 = \frac{n+2}{5},
\]
which completes the proof of (i).

To see that this bound is sharp, consider the following family.  Starting with a tree $T$ on $2p$ vertices all having degree one or three, subdivide each pendant edge twice and call this tree $T^*$.  Clearly, in both $T$ and $T^*$, $n_1=p+1$ and $n_3=p-1$.  On the other hand, in $T^*$, $n_2 = 2n_1 = 2p+2$.  Thus, $T^*$ has a total of $4p+2$ vertices, where the degree two vertices induce a matching.  Now add $p+1$ isolated vertices to make a forest $F$ of order $n=5p+3$ where $n_0=n_1=\frac{n_2}{2}=n_3+2=p+1 = \frac{n+2}{5}$. Since $n_3 < p+1$, the degree zero vertices are independent, the degree one vertices are independent, and exactly half of the degree two vertices are independent, $\alpha_{reg}(F)=\frac{n+2}{5}$ -- which shows the bound is sharp (see $F_1$ in Figure \ref{fig-forests}).

To prove (ii), notice that if either $n_0 \geq \frac{2(n+2)}{9}$ or $n_1 \geq \frac{2(n+2)}{9}$, we are done since both of the subgraphs induced by the vertices of degree zero and one respectively are $1$-independent sets and $\onereg(F) \geq \max \{\alpha_{1,0},\alpha_{1,1}\}=\max \{ n_0,n_1\}$.  So we may assume that $n_0 < \frac{2(n+2)}{9}$ and $n_1 < \frac{2(n+2)}{9}$.  As in the proof of the previous proposition, we use the observations that $n_1 \geq N_3+2$ and $n_2 = n-n_0-n_1-N_3$ to deduce that $n_2 \geq n-n_0-2n_1+2$.  Finally, we use the fact that $\alpha_{1,2}(F) \geq \frac{2}{3}n_2$ since the degree two vertices induce a collection of paths to complete the proof as follows:
\begin{align*}
\onereg(F) & \geq \alpha_{1,2}(F) \geq \frac{2}{3}n_2 \geq \frac{2}{3}(n -n_0 - 2n_1 +2) \\
& > \frac{2}{3}(n - \frac{2(n+2)}{9} - \frac{4(n+2)}{9} +2) = \frac{2(n+2)}{9}.
\end{align*}
To see that this bound is sharp, let $p$ be a positive integer.  We will construct a forest $F$ with $n=9p+7$ vertices where equality holds.  First, observe that there is a tree with exactly $\frac{4(n+2)}{9}-2$ vertices all of which having degree one or degree three.  To see this is true, note that we can start with a double star on six vertices (i.e. two $K_{1,2}$ with the central vertices joined by an edge) with $n_1=4$ and $n_3=2$, corresponding to the case that $n=16$, and then attach two leaves to each of two existing leaves each time $p$ is incremented. The trees formed this way will always have $4p+2=\frac{4(n+2)}{9}-2$ vertices with $2p+2$ leaves and $2p$ degree three vertices.  Now take a path on $\frac{3(n+2)}{9}$ vertices and join an endpoint of this path to a leaf of the previously formed tree.  Finally, add $\frac{2(n+2)}{9}$ isolated vertices.  This graph is a forest with exactly
\[
\frac{4(n+2)}{9}-2 + \frac{3(n+2)}{9} + \frac{2(n+2)}{9} = n
\]
vertices such that $n_0=\frac{2(n+2)}{9}$, $n_1=\frac{2(n+2)}{9}$, $n_2=\frac{3(n+2)}{9}$, and $N_3=n_3=\frac{2(n+2)}{9}-2$.  Since the degree two vertices induce a path whose order is a multiple of three, and the $1$-independence number of such paths is easily seen to be $\frac{2}{3}$ of its order, we find that $\alpha_{1,2}(F)=\frac{2(n+2)}{9}$.  Moreover, since the degree zero and the degree one vertices induce independent sets, and hence $1$-independent sets, we also find that $\alpha_{1,0}=\alpha_{1,1}=\frac{2(n+2)}{9}$.  Together with the fact that $\alpha_{1,3}\leq n_3 < \frac{2(n+2)}{9}$, we conclude $\onereg(F) = \frac{2(n+2)}{9}$, which shows the bound is sharp (see $F_2$ in Figure \ref{fig-forests}).

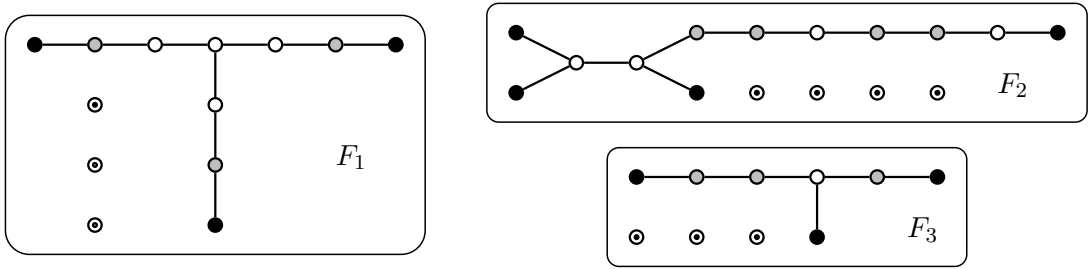
\begin{figure}[h]
\begin{center}
\psset{unit=0.8cm, linewidth=0.03cm}
   \begin{pspicture}(-0.8,-2)(18,3.7)
      \psframe[framearc=0.2,linewidth=0.02cm](-0.5,-1.5)(6.5,2.5)
      \put(5,0){$F_1$}
      
      \cnode*(0,2){0.13}{a1}
      \cnode[fillstyle=solid,fillcolor=lightgray](1,2){0.13}{a2}
      \cnode(2,2){0.13}{a3}
      \cnode(3,2){0.13}{a4}
      \cnode(4,2){0.13}{a5}
      \cnode[fillstyle=solid,fillcolor=lightgray](5,2){0.13}{a6}
      \cnode*(6,2){0.13}{a7}
      \cnode(3,1){0.13}{a8}
      \cnode[fillstyle=solid,fillcolor=lightgray](3,0){0.13}{a9}
      \cnode*(3,-1){0.13}{a10}
      \cnode(1,1){0.13}{a11}
      \cnode*(1,1){0.05}{a11bis}
      \cnode(1,0){0.13}{a12}
      \cnode(1,0){0.05}{a12bis}
      \cnode(1,-1){0.13}{a13}
      \cnode(1,-1){0.05}{a13bis}
      
      \ncline{a1}{a2}
      \ncline{a2}{a3}
      \ncline{a3}{a4}
      \ncline{a4}{a5}
      \ncline{a5}{a6}
      \ncline{a6}{a7}
      \ncline{a4}{a8}
      \ncline{a8}{a9}
      \ncline{a9}{a10}
     
      \psframe[framearc=0.2,linewidth=0.02cm](7.5,0.7)(17.5,2.7)
      \put(16,1.2){$F_2$}
      
      \cnode*(8,2.2){0.13}{b1}
      \cnode*(8,1.2){0.13}{b2}
      \cnode(9,1.7){0.13}{b3}
      \cnode(10,1.7){0.13}{b4}
      \cnode[fillstyle=solid,fillcolor=lightgray](11,2.2){0.13}{b5}
      \cnode*(11,1.2){0.13}{b6}
      \cnode[fillstyle=solid,fillcolor=lightgray](12,2.2){0.13}{b7}
      \cnode(13,2.2){0.13}{b8}
      \cnode[fillstyle=solid,fillcolor=lightgray](14,2.2){0.13}{b9}
      \cnode[fillstyle=solid,fillcolor=lightgray](15,2.2){0.13}{b10}
      \cnode(16,2.2){0.13}{b11}
      \cnode*(17,2.2){0.13}{b12}
      \cnode(12,1.2){0.13}{b13}
      \cnode(13,1.2){0.13}{b14}
      \cnode(14,1.2){0.13}{b15}
      \cnode(15,1.2){0.13}{b16}
      \cnode*(12,1.2){0.05}{b13}
      \cnode*(13,1.2){0.05}{b14}
      \cnode*(14,1.2){0.05}{b15}
      \cnode*(15,1.2){0.05}{b16}
      
      \ncline{b1}{b3}
      \ncline{b2}{b3}
      \ncline{b3}{b4}
      \ncline{b4}{b5}
      \ncline{b4}{b6}
      \ncline{b5}{b7}
      \ncline{b7}{b8}
      \ncline{b8}{b9}
      \ncline{b9}{b10}
      \ncline{b10}{b11}
      \ncline{b11}{b12}
      
       \psframe[framearc=0.2,linewidth=0.02cm](9.5,-1.7)(15.5,0.3)
       \put(14.5,-1.2){$F_3$}
       
      \cnode*(10,-0.2){0.13}{c1}
      \cnode[fillstyle=solid,fillcolor=lightgray](11,-0.2){0.13}{c2}
      \cnode[fillstyle=solid,fillcolor=lightgray](12,-0.2){0.13}{c3}
      \cnode(13,-0.2){0.13}{c4}
      \cnode[fillstyle=solid,fillcolor=lightgray](14,-0.2){0.13}{c5}
      \cnode*(15,-0.2){0.13}{c6}
      \cnode*(13,-1.2){0.13}{c7}
      \cnode(10,-1.2){0.13}{c8}
      \cnode(11,-1.2){0.13}{c9}
      \cnode(12,-1.2){0.13}{c10}
      \cnode*(10,-1.2){0.05}{c8}
      \cnode*(11,-1.2){0.05}{c9}
      \cnode*(12,-1.2){0.05}{c10}
      
      \ncline{c1}{c2}
      \ncline{c2}{c3}
      \ncline{c3}{c4}
      \ncline{c4}{c5}
      \ncline{c5}{c6}
      \ncline{c4}{c7}

\end{pspicture}
\caption{ \label{fig-forests} Forests satisfying equality in Theorem \ref{thm-forests}. Monochromatic vertices, aside from the vertices colored white, form different $\reg(F_1)$-sets, $\onereg(F_2)$-sets and  $\kreg(F_3)$-sets in each case.}
\end{center}
\end{figure}

To prove (iii), let $k \geq 2$ and notice that if either $n_0 \geq \frac{n+2}{4}$ or $n_1 \geq \frac{n+2}{4}$, we are done since both of the subgraphs induced by the vertices of degree zero and one respectively are $k$-independent sets and $\alpha_{k-reg}(F) \geq \max \{\alpha_{k,0},\alpha_{k,1}\}=\max \{ n_0,n_1\}$.  So we may assume that $n_0 < \frac{n+2}{4}$ and $n_1 < \frac{n+2}{4}$.  We again make use of $n_2 \geq n-n_0-2n_1+2$ together with $\alpha_{k,2}=n_2$ to complete the proof as follows:
\[
\alpha_{k-reg}(F) \geq \alpha_{k,2}(F) = n_2 \geq n -n_0 - 2n_1 +2 > n - \frac{n+2}{4} - \frac{2(n+2)}{4} +2 = \frac{n+2}{4}.
\]
To see that this bound is sharp, start with a path on $q\geq 1$ vertices and attach a path of length two to each vertex.  Now attach an additional leaf to one endpoint of the initial path and add $q+1$ isolated vertices.  This family has the following properties; $n_0=n_1=n_2=n_3+2=q+1$, $n=4q+2$ and $\alpha_{k-reg} = \alpha_{k,0} = \alpha_{k,1} = \alpha_{k,2} = q+1 = \frac{n+2}{4}$ -- which shows the bound is sharp (see $F_3$ in Figure \ref{fig-forests}).
\end{proof}

The theorems above show that, for trees and forests, finding $\kreg(G)$ can be accomplished by finding the independence number of the subgraphs induced by the vertices of degrees $2$, degree $1$ and degree $0$ (in the case of forests). This can be easily performed by a linear time algorithm \cite{GaJo}.

To end this section, we would like to mention that the bounds given here for trees and forests can be extended easily to $n$-vertex graphs with $o(n)$ cycles. This can be done by deleting one by one the edges from the $o(n)$ cycles until one is left with a tree or, in case the original graph was disconnected, with a forest. Then, a regular $k$-independent set of the  resulting tree/forest can be transformed into a regular $k$-independent set of the original graph by deleting only the $o(n)$ vertices whose degree were affected in the process of deleting edges. Doing so, we obtain a regular $k$-independent set whose size is $(1-o(1))$ times the size of the regular $k$-independent set of the tree/forest.

\section{$k$-trees, $k$-degenerate graphs and planar graphs}

In this section, we will refine the method used by Alberson and Boutin in \cite{AlBo} and we will give lower bounds on $\reg$ for $k$-trees, $k$-degenerate graphs and planar graphs. Due to the structural characteristics of $k$-trees, we are able to get a better insight and hence stronger results than for $k$-degenerate graphs and planar graphs. There are straightforward but meticulous analogs to the procedures we use below for $k$-trees, in the cases of $k$-degenerate graphs and planar graphs, and, for this reason, we only present the proofs for $k$-trees, while summarizing the results we obtained for the other graph families in Table \ref{table-og} in what follows.

Throughout this and the next section, we will use $V_i(G)$ to denote the set of vertices of degree $i$ in $G$ and $n_i(G) = |V_i(G)|$. When the context is clear, $V_i(G)$ will be abbreviated to $V_i$ and $n_i(G)$ to $n_i$. Moreover, for any set $S \subseteq V(G)$, we will write $\chi_k(S)$ instead of $\chi_k(G[S])$.  Recall that a \emph{$k$-tree} may be formed by starting with a  complete graph $K_{k+1}$ and then adding repeatedly vertices in such a way that each added vertex has exactly $k$ neighbors that form a clique. A \emph{$k$-degenerate graph} is a graph all of whose induced subgraphs have minimum degree at most $k$. A \emph{maximal $k$-degenerate graph} is a $k$-degenerate graph with the maximum possible number of edges, which is $kn - \frac{k(k+1)}{2}$, where $n$ is the order of the graph. Note also that a (maximal) $k$-degenerate graph may be formed by starting with a vertex (a $(k+1)$-clique) and then repeatedly adding vertices in such a way that each added vertex has at most (exactly) $k$ neighbors. In particular, a $k$-tree is a maximal $k$-degenerate graph (but not necessarily  vice versa). Note also that every $k$-degenerate graph has a proper coloring using $k+1$ colors. The next two lemmas and Corollary \ref{cor-k-trees} are used to prove Theorem \ref{k-trees}, which is a main result of our paper.

\begin{lem}\label{lem-n_k+t}
Let $G$ be a $k$-tree on $n=k+t+2$ vertices, where $t \ge 1$ is an integer. Then $n_{k+t} \le t+1$.
\end{lem}
\begin{proof}
We will prove the statement by induction on $t$. If $t=1$ and $n = k+3$, it is not difficult to see that there are only two non-isomorphic $k$-trees on $k+3$ vertices and that they have at most $2$ vertices of degree $k+1$. Hence $n_{k+1} \le 2$ and the base case is settled. Suppose that for any $k$-tree on $n-1 = k + (t-1) + 2$ vertices, where $t-1 \ge 1$, there are at most $t$ vertices of degree $k+t-1$. Let $G$ be a $k$-tree on $n = k+t+2$ vertices and let $x$ be a vertex of degree $k$ in $G$. Let $G^* = G - x$. Then $n(G^*) = k+(t-1)+2$ and, by induction, $G^*$ has at most $t$ vertices of degree $k+t-1$. Note that $N(x)$ is a clique. If there are two vertices $u , v \in V(G) \setminus N[x]$ of degree $k+t$ in $G$, then $u$ and $v$ have to be adjacent to each other and to every other vertex different from $x$, implying that $N(x) \cup \{u,v\}$ is a clique of order $k+2$, which is a contradiction as $G$ is a $k$-tree with maximum clique order $k+1$. Hence, there is at most one vertex of degree $k+t$ in $V(G) \setminus N[x]$ and all other vertices of degree $k+t$ in $G$ have to be contained in $N(x)$. Since the vertices of degree $k+t$ in $G$ contained in $N(x)$ have degree $k+t-1$ in $G^*$, there are at most $t$ of them by inductive assumption. Thus, in $G$ there are at most $t+1$ vertices of degree $k+t$.
\end{proof}

\begin{lem}\label{lem-chi}
Let $G$ be a $k$-tree of order $n \ge k+t+2$, where $t \ge 0$ is an integer. Then 
$$\chi(V_{k+t}(G)) \le \frac{1}{2}(t^2+t+2).$$
\end{lem}
\begin{proof}
The proof goes by induction on $t$. If $t=0$, then, since $n \ge k+2$, the vertices of degree $k$ form an independent set and thus we have clearly $\chi(V_k(G)) \le 1 = \frac{1}{2}(t^2+t+2)$. This settles the base case. Let $t \ge 1$. Assume that 
\begin{equation}\label{ind-assump}
\chi(V_{k+t-1}(G)) \le \frac{1}{2}((t-1)^2+(t-1)+2) = \frac{1}{2}(t^2-t+2)
\end{equation}
for every  $k$-tree $G$ on $n \ge k+(t-1)+2 = k+t+1$ vertices. Now we will show that $\chi(V_{k+t}(G)) \le \frac{1}{2}(t^2+t+2)$ for every $k$-tree $G$ on $n \ge k+t+2$ vertices. Again, we use induction but now on $n$. 

If $G$ is a $k$-tree on $n = k+t+2$ vertices, then Lemma \ref{lem-n_k+t} yields $n_{k+t}(G) \le t+1$ and we obtain $\chi(V_{k+t}(G)) \le t+1 \le \frac{1}{2}(t^2+t+2)$, which settles the base case. Assume now that $\chi(V_{k+t}(G)) \le \frac{1}{2}(t^2+t+2)$ for every $k$-tree $G$ on $n$ vertices, where $n$ is a fixed integer with $n \ge k+t+2$. Let $G$ be a $k$-tree on $n+1$ vertices. We will show that $\chi(V_{k+t}(G)) \le \frac{1}{2}(t^2+t+2)$ also holds.  Let $x \in V(G)$ be a vertex of degree $k$ and let $G^* = G - x$. Then $n(G^*) = n$, $\chi(V_{k+t-1}(G^*)) \le \frac{1}{2}(t^2-t+2)$ because of (\ref{ind-assump}) and,  by induction, $\chi(V_{k+t}(G^*)) \le \frac{1}{2}(t^2+t+2)$. Then, since $N_G(x)$ is a clique, $|V_{k+t-1}(G^*) \cap N_G(x)| \le \frac{1}{2}(t^2-t+2)$. Note that the vertices of degree $k+t$ in $G$ are either contained in $V(G)\setminus N_G[x]$ and have degree $k+t$ in $G^*$ or they are contained in $N_G(x)$ and have degree $k+t-1$ in $G^*$. Observe also that every vertex $u \in V_{k+t}(G) \cap N_G(x)$ has exactly $t$ neighbors in $V(G)\setminus N_G[x]$ and hence it has at most $t$ neighbors in $V_{k+t}(G)\setminus N_G[x]$. Thus, we can transform a proper coloring of $V_{k+t}(G^*)$ with $ \frac{1}{2}(t^2+t+2)$ colors into a proper coloring of $V_{k+t}(G)$ with $\frac{1}{2}(t^2+t+2)$ colors the following way. Every vertex contained in $V_{k+t}(G) \setminus N_G[x] = V_{k+t}(G^*) \setminus N_G[x]$ remains colored with the same color. Since each vertex of $V_{k+t}(G) \cap N_G(x)$ has at most $t$ neighbors in $V_{k+t}(G) \setminus N_G[x]$ and since $|V_{k+t}(G) \cap N_G(x)| = |V_{k+t-1}(G^*) \cap N_G(x)| \le \frac{1}{2}(t^2-t+2)$, it follows that the vertices of $V_{k+t}(G) \cap N_G(x)$ have at most $\frac{1}{2}(t^2-t) - 1 + t = \frac{1}{2} (t^2+t)<  \frac{1}{2}(t^2+t+2)$ neighbors of degree $k+t$ in $G$. Hence, coloring the vertices of $V_{k+t}(G) \cap N_G(x)$ one after the other, we can assign each vertex a color that does not  appear on its neighbors.

 Hence, it follows by induction that $\chi(V_{k+t}(G)) \le \frac{1}{2}(t^2+t+2)$ for every $k$-tree on $n \ge k+t+2$ vertices.
\end{proof}

\begin{rem}
Note that in Lemma \ref{lem-chi} we prove something stronger; namely that, given a $k$-tree $G$ of order $n \ge k +t +2$, $t \ge 0$, the graph $G[V_{k+t}(G)]$ is $\frac{1}{2}(t^2 +t)$-degenerate. Hence, clearly $\chi(V_{k+t}) \le \frac{1}{2}(t^2+t+2)$. 
\end{rem}

Now, via Lemma \ref{lem-chi}, we are able to bound from above the number of vertices of equal degree in a $k$-tree. However, this makes sense only for the values of $t$ where $n_{k+t}(G) \le (k+1) \reg(G)$ is not a better bound. We will use a parameter $q(k)$ to mark this critical point.

\begin{cor}\label{cor-k-trees}
Let $G$ be a $k$-tree on $n \ge k+q(k)+2$ vertices, where $q(k) = \lfloor \frac{-1+\sqrt{8k-7}}{2} \rfloor$. Then 
\[n_{k+t}(G) \le \frac{1}{2} (t^2+t+2) \reg(G), \mbox{ for } 0 \le t \le q(k), \mbox{ and}\]
\[ n_{k+t}(G) \le (k+1) \reg(G), \mbox{ for } t \ge q(k)+1.\]
\end{cor}

\begin{proof}
Let $q(k)$ be the maximum integer for which $\frac{1}{2}(q(k)^{2}+q(k)+2) \le k$. Then $q(k) = \lfloor \frac{-1+\sqrt{8k-7}}{2} \rfloor$. Since $n_{k+t}  \le \chi(V_{k+t}(G)) \reg(G)$, Lemma \ref{lem-chi} yields 
\[n_{k+t}(G) \le \frac{1}{2}(t^2+t+2)  \reg(G)\]
for $0 \le t \le q(k)$. For $t \ge q(k)+1$, we use the fact that every $k$-tree on $n \ge k+3$ vertices is $(k+1)$-colorable and we obtain $n_{k+t}  \le \chi(V_{k+t}(G)) \reg(G) \le (k+1) \reg(G)$.
\end{proof}

Corollary \ref{cor-k-trees} enables us to compute a bound on $\reg(G)$ for $k$-trees, which will be given in the next theorem. The proof of this theorem, adapted for each particular case, contains the essence of the computations of all other bounds that are given further on for planar graphs, outerplanar graphs and $k$-degenerate graphs -- and is therefore a main result of our paper.

\begin{thm}\label{k-trees} 
Let $G$ be a $k$-tree on $n \ge k+q(k)+2$ vertices, where $q(k) = \lfloor \frac{-1+\sqrt{8k-7}}{2} \rfloor$ and $k \ge 2$. Then 
\[\reg(G) > \frac{24k}{48k^3 + 84k^2 - 72k  - (16k^2-13)\sqrt{8k-7} + 36}\; n.\]
\end{thm}

\begin{proof}
Since $G$ is a $k$-tree on $n$ vertices, $e(G) = kn - \frac{k(k+1)}{2}$. Hence, we have
\[\sum_{i\ge k} i n_i = 2 \left(kn - \frac{k(k+1)}{2}\right) = 2kn - k(k+1) = \sum_{i \ge k} 2kn_i - k(k+1)\]
and thus
\[\sum_{i \ge 2k+1} (i-2k) n_i = \sum_{i=k}^{2k} (2k-i)n_i- k(k+1).\]
This implies, for any index $r \ge 2k+1$,
\begin{align*}
(r-2k) \sum_{i \ge r} n_i \;\le\; \sum_{i \ge r} (i-2k) n_i &= \sum_{i \ge 2k+1} (i-2k) n_i - \sum_{i =2k+1}^{r-1} (i-2k) n_i \\
                                                                                     &= \sum_{i=k}^{r-1} (2k-i)n_i -k(k+1),
\end{align*}
yielding $$\sum_{i \ge r} n_i \; \le\; \frac{1}{r-2k} \left(\sum_{i=k}^{r-1} (2k-i)n_i -k(k+1)\right).$$
Thus, for any $r \ge 2k+1$, we have
\begin{align*}
n\; =\; \sum_{i \ge  k} n_i &= \sum_{i = k}^{r-1} n_i + \sum_{i \ge r} n_i\\
                                      &\le \sum_{i = k}^{r-1} n_i + \frac{1}{r-2k} \left(\sum_{i=k}^{r-1} (2k-i)n_i -k(k+1) \right)\\ 
                                      &= \frac{1}{r-2k} \left( \sum_{i=k}^{r-1} (r-i) n_i - k(k+1) \right).
\end{align*}
We will now use the inequalities given in Corollary \ref{cor-k-trees} to bound the above inequality. Let $q(k) = \lfloor \frac{-1+\sqrt{8k-7}}{2} \rfloor$.
\begin{align*}
n &\le \frac{1}{r-2k} \left( \left(\sum_{i=k}^{r-1} (r-i) n_i \right)- k(k+1) \right)\\
&<  \frac{1}{r-2k} \sum_{i=k}^{r-1} (r-i) n_i\\
   &= \frac{1}{r-2k} \left( \left( \sum_{i=k}^{q(k)+k} (r-i) n_i \right)+ \left(\sum_{i=q(k)+k+1}^{r-1} (r-i) n_i \right) \right)\\
   &\le \frac{\reg(G)}{r-2k}  \left( \left(\sum_{i=k}^{k+q(k)} \frac{1}{2}(r-i)((i-k)^2+(i-k)+2) \right) + \left(\sum_{i=k+q(k)+1}^{r-1} (r-i)(k+1) \right) \right)\\
  &  = \reg(G) \cdot \frac{a_k(r)+b_k(r)}{r-2k},
      \end{align*}
where
\begin{align*}
a_k(x) & = \sum_{i=k}^{k+q(k)} \frac{1}{2}(x-i)((i-k)^2+(i-k)+2), \; \mbox{and}\\
b_k(x) & = \left(\sum_{i=k+q(k)+1}^{x-1} (x-i)(k+1) \right)\\
& = \frac{k+1}{2} \left(x^2-(1+2k+2q(k))x+(k+q(k))^2+k+q(k)\right)
\end{align*}   
are functions defined for $x \in [2k+1, \infty)$.\\
Define the function $f_k(x) = \frac{a_k(x)+b_k(x)}{x-2k}$ for $x \in [2k+1, \infty)$. By the above computation, we have
\begin{equation}\label{ineq-n-f}
n < \reg(G) f_k(r).
\end{equation}
The minimum of the function $f_k(x)$ is attained when 
$$x = \frac{1}{(k+1)}\left( 2k^2+2k+ \sqrt{k^4+ R(k)}\right),$$
where $R(k) = \mathcal{O}(k^{7/4})$. Hence, the minimum of the function has order $\mathcal{O}(3k)$. Therefore, we set $r = 3k$ and calculate $f_k(3k)$, which gives us
\begin{align*}
f_k(3k) =& \;\frac{1}{24k}\left(48k^3+24k^2+24k-3q(k)^4+(-10+8k)q(k)^3 \right.\\
              & \left. +(36k-9)q(k)^2-(48k^2+2-28k)q(k) \right).
\end{align*}
Doing so, we will give a slightly weaker bound than if we used the optimal value of the function, but it is much easier to calculate and suffices for our purpose.
As $q_1(k) = \frac{-3+\sqrt{8k-7}}{2} \le q(k)= \lfloor \frac{-1+\sqrt{8k-7}}{2} \rfloor \le \frac{-1+\sqrt{8k-7}}{2} = q_2(k)$ and since the coefficients $(-10+8k)$, $(36k-9)$ and $(48k^2+2-28k)$ are all positive for $k \ge 2$, we obtain
\begin{align*}
f_k(3k) \le& \;\frac{1}{24k}(48k^3+24k^2+24k-3q_1(k)^4+(-10+8k)q_2(k)^3+(36k-9)q_2(k)^2\\
        & \;-(48k^2+2-28k)q_2(k))\\
         =& \;\frac{48k^3 + 84k^2 - 72k  - (16k^2-13)\sqrt{8k-7} + 36}{24k}.
\end{align*}
We conclude the proof by substituting this into Inequality (\ref{ineq-n-f}), yielding
$$\reg(G) > \frac{24k}{48k^3 + 84k^2 - 72k  - (16k^2-13)\sqrt{8k-7} + 36}\; n.$$
\end{proof}

In Table \ref{table-k-trees}, we give more accurate lower bounds on the regular independence number of $k$-trees of order $n$  with $1 \le k \le 10$. This is done by a more detailed analysis of the function $f_k(x) $ used in the proof of Theorem \ref{k-trees} and by calculating the bound from the inequality $n \le \reg(G) f_k(r) - \frac{k(k+1)}{r-2k}$. For instance, for $k=2$, we have $f_2(x) = \frac{3x^2-15x+20}{2(r-4)}$ for which $x = 6$ is the integer that is closest to its minimum. This yields $n \le \reg(G) f_2(6) - 3 = \reg(G) \frac{19}{2} - 3$ and, hence, we obtain the lower bound $\reg(G) \ge \frac{2}{19}(n+3)$ for a $2$-tree $G$ on $n$ vertices.

  \begin{figure}[!h]
    \begin{center}
    \begin{tabular}{|c||c|}
        \hline
        $k$   &     Bound \\
        \hline\hline
        $1$ & $\frac{1}{4}(n+2)$\\ \hline
        $2$ & $\frac{2}{19}(n+3)$ \\ \hline
        $3$ & $\frac{2}{37}(n+6)$ \\ \hline 
        $4$ & $\frac{3}{89}(n+\frac{20}{3})$ \\ \hline
        $5$ & $\frac{4}{179}(n+\frac{15}{2})$\\
                \hline
                  \end{tabular}
       \hspace{2ex}
       \begin{tabular}{|c||c|}
        \hline
        $k$   &     Bound \\
        \hline\hline
         $6$  & $\frac{5}{319}(n+\frac{42}{5})$\\ \hline
         $7$ & $\frac{1}{85}(n+\frac{56}{5})$ \\ \hline
         $8$ & $\frac{1}{110}(n+12)$ \\ \hline
         $9$ & $\frac{1}{139}(n+\frac{90}{7})$\\ \hline
         $10$ & $\frac{1}{172}(n+\frac{55}{4})$\\
                \hline
    \end{tabular}
   \captionof{table}{Lower bounds on $\reg(G)$ for $k$-trees, $1 \le k \le 10$.} \label{table-k-trees}
   \end{center}
\end{figure}

We remark that, when $k=1$, the bound given in Table \ref{table-k-trees} is precisely the bound for trees given in \cite{AlBo} (and in Theorem \ref{thm-trees} (i)).

In a similar way as for the $k$-trees, we can compute lower bounds on $\reg(G)$ for planar graphs, outerplanar graphs and $k$-degenerate graphs. The following lemmas and the next corollary give us the needed background theory to achieve this goal.  Recall that a \emph{maximal planar graph} is a planar graph with the maximum possible number of edges; namely $3n-6$ where $n \ge 3$ is the order of the graph. Further, a \emph{(maximal) outerplanar graph} is a triangulation of a polygon and it has at most (exactly) $2n-3$ edges, where, again, $n \ge 3$ is the number of vertices.

\begin{lem}
Let $G$ be a connected graph of order $n$ which is not an odd cycle or a complete graph. Then
$\chi(V_i(G)) \le \min\{i, \chi(G)\}$.
\end{lem}
 
\begin{proof}
If $\chi(G) \le i$, then evidently $\chi(V_i) \le \chi(G) \le i$ and we are done. Hence we may assume that $\chi(G) \ge i$. Since $\Delta(G[V_i]) \le i$, Brooks' Theorem \cite{Br} implies that $\chi(V_i) = i+1$ if $G[V_i]$ contains either a component which is a complete graph on $i+1$ vertices or $i=2$ and $G[V_2]$ contains a component which is an odd cycle. Both cases are impossible as $G$ is connected and is neither a complete graph nor an odd cycle. Hence $\chi(V_i) \le i$.
\end{proof}

\begin{cor}\label{cor_p-op-d}
Let $G$ be a connected graph of order $n$ and which is not an odd cycle. Then the following statements hold: \\[-4ex]
\begin{enumerate}
\item[(1)] If $G$ is planar and $n \ge 5$, then $\chi(V_i(G)) \le i$ for $i = 1, 2, 3$ and $\chi(V_i) \le 4$ for $i \ge 4$.
\item[(2)] If $G$ is outerplanar and $n \ge 4$, then $\chi(V_1(G)) =1$, $\chi(V_2(G)) \le 2$ and $\chi(V_i) \le 3$ for $i \ge 3$.
\item[(3)]  If $G$ is $k$-degenerate and $n \ge k+2$, then $\chi(V_i(G)) \le i$ for $i \le k$ and $\chi(V_i) \le k+1$ for $i \ge k+1$.
\end{enumerate}
\end{cor}

When $G$ is a maximal planar, maximal outerplanar or a maximal $k$-degenerate graph, we also know the following facts:

\begin{lem}\label{lem_max-p-op-d}
The following statements are valid:\\[-4ex]
\begin{enumerate}
\item[(1)] If $G$ is a maximal planar graph on $n \ge 5$ vertices, then $\chi(V_3(G)) = 1$, $\chi(V_4(G)) \le 3$ and $\chi(V_i(G)) \le 4$ for $i \ge 5$ (see \cite{AlBo}).
\item[(2)] If $G$ is a maximal outerplanar graph on $n \ge 3$ vertices, then $\chi(V_3(G)) \le 2$ and $\chi(V_i(G)) \le 3$ for $i \ge 4$.
\item[(3)] If $G$ is a maximal $k$-degenerate graph on $n \ge k+2$ vertices, then $\chi(V_k(G))=1$ and $\chi(V_i(G)) \le k+1$ for $i \ge k+1$.
\end{enumerate}
\end{lem}

\begin{proof}
(1) The proof of this item is given in \cite{AlBo}, but for completeness we present it here again. In a maximal planar graph on $n \ge 5$ vertices, the vertices of degree $3$ are independent. Also it is easy to check that each component of $G[V_4(G)]$ is either $K_3$-free or is a $K_3$. Since, by Gr\"otsch's Theorem, $K_3$-free planar graphs are $3$-colorable, $G[V_4(G)]$ is $3$-colorable. Moreover, $\chi(V_i(G)) \le \chi(G) \le 4$ by the Four-Color-Theorem \cite{ApHa, ApHaKo}.\\
(2) The cases $n=3$ and $n=4$ are trivial. For $n \ge 5$,  $\chi(V_3(G)) \le 2$ follows from Lemma \ref{lem-chi} and the fact that maximal outerplanar are $2$-trees (see also \cite{CHH} for an explicit proof). Moreover, since maximal outerplanar graphs are $3$-colorable, we have $\chi(V_i(G)) \le 3$ for $i \ge 4$.\\ 
(3) Since in a  maximal $k$-degenerate graph $G$ the vertices of degree $k$ are independent and $\chi(G) \le k+1$, we have  $\chi(V_k(G))=1$ and $\chi(V_i(G)) \le \chi(G) \le k+1$ for $i \ge k+1$.
\end{proof}

Using $n_i \le \reg(G) \chi(V_i)$ together with Corollary \ref{cor_p-op-d} and Lemma \ref{lem_max-p-op-d}, we can bound the number of vertices of degree $i$  for connected planar graphs, outerplanar graphs and $k$-degenerate graphs. Proceeding as we did with the $k$-trees in Theorem \ref{k-trees}, we can find lower bounds on the regular independence number for each of these graph types. Our results are listed in Table \ref{table-og}.

\begin{figure}[!h]
\begin{center}
    \begin{tabular}{|c|c|c|c|}
        \hline
                &   Bench Mark   &  Bound obtained using  &    Bound provided\\[-1ex]
                & from Prop. \ref{prop1} & Corollary \ref{cor_p-op-d} and Lemma \ref{lem_max-p-op-d} & Edge Maximality  \\
                \hline
                \hline
         \multicolumn{4}{|c|}{{\bf Connected planar graphs}}\\ 
         \hline
        $\delta = 1$ & $\frac{1}{44}n$ & $\frac{2}{65}(n+3)$ & --\\ \hline
        $\delta = 2$ & $\frac{1}{36}n$ & $\frac{4}{121}(n+3)$ & --\\ \hline
        $\delta = 3$ & $\frac{1}{28}n$ & $\frac{1}{26}(n+4)$ & $\frac{3}{61}(n+4)$\\ \hline
        $\delta = 4$ & $\frac{1}{20}n$ & $\frac{1}{20}(n+6)$ & $\frac{1}{18}(n+6)$\\ \hline
        $\delta = 5$ & $\frac{1}{12}n$ & $\frac{1}{12}(n+12)$ & $\frac{1}{12}(n+12)$\\  \hline
               \hline
        \multicolumn{4}{|c|}{{\bf Connected outerplanar graphs}}\\
        \hline
        $\delta = 2$ & $\frac{1}{15}n$ & $\frac{1}{13}(n+3)$ & $\frac{2}{19}(n+3)$\\ \hline
               \hline
        \multicolumn{4}{|c|} {{\bf Connected $k$-degenerate graphs}}\\
        \hline
        $\delta < k$ & $\frac{1}{8k^2-(2 \delta -1)k+1}n$ & $\frac{12n+6(k+1)}{37k^2+27k+12 \delta - 12 \delta^2-10+\frac{2\delta^3-3\delta^2+\delta}{k}}$ & --\\ \hline
        $\delta= k$ & $\frac{1}{6k^2+k+1}n$ & $\frac{n+k+1}{2k^2+3k-1}$ & $\frac{n+k+1}{2k^2+k+1}$ \\ \hline
    \end{tabular}
    \captionof{table}{Lower bounds on $\reg(G)$} \label{table-og}
        \end{center}
    \end{figure}

We remark that, while the bound on planar graphs with $\delta =1$ and the bound on maximal planar graphs with $\delta = 3$ are only very tiny refinements of Albertson and Boutin's results (1) and (2) mentioned in the introduction of this paper, the other bounds on planar graphs improve upon them considerably. Further, although for general planar graphs with $\delta = 4$ and $\delta = 5$ the improvement is modest, all bounds obtained via the procedure of Theorem \ref{k-trees} are better than the benchmark bound from Proposition \ref{prop1}. Note also that the bounds on $k$-degenerate graphs with $\delta = k$ and on maximal $k$-degenerate graphs generalize Alberson and Boutin's bound for trees (see (3) in the introduction and Proposition \ref{thm-trees} (i)). Furthermore, we remark that the bound on maximal outerplanar graphs is the complementary result to the one about the fair domination number obtained in \cite{CHH} and the same bound from Table \ref{table-k-trees} when $k=2$ (derived from the fact that maximal outerplanar graphs are a special kind of $2$-trees).

Observe that it is important to have connected graphs for the bounds given in Table~\ref{table-og}, since in general it is not true that, for two disjoint graphs $G$ and $H$,  $\reg(G \cup H)  = \reg(G) + \reg(H)$, as $\reg(G)$ and $\reg(H)$ could be attained by sets of vertices each with a different degree. This phenomena already occurred in the case of trees versus forests in the previous section.

\section{Bounds on $\tworeg(G)$ for planar graphs}

In this section, we present some lower bounds on $\tworeg(G)$ for planar and outerplanar graphs $G$. Recall that the $k$-chromatic number $\chi_k(G)$ is the minimum number of colors needed to color the vertices of the graph $G$ such that the graphs induced by the vertices of each color class have maximum degree at most $k$. In \cite{Lov}, Lov\'asz shows that
\[
\chi_k(G) \le  \left \lceil \frac{\Delta(G) +1}{k+1} \right\rceil.
\]

Dealing with the more general concept of defective colorings, also called improper colorings (see \cite{CoCoWo, CoGoJe, Fri, Rack}), Cowen, Cowen and Woodall \cite{CoCoWo} show that $\chi_2(G) \le 2$ for any outerplanar graph $G$. Moreover, they prove that there are outerplanar graphs $G$ with $\chi_1(G) = 3$, they show that $\chi_2(G) \le 3$  for all planar graphs $G$ and, finally, that there are planar graphs $G$ with $\chi_1(G) = 4$.

\begin{lem}\label{lem-p-o-2}
The following statements hold:\\[-4ex]
\begin{enumerate}
\item[(1)] If $G$ is a planar graph, then $\chi_2(V_1) = \chi_2(V_2) = 1$, $\chi_2(V_i) \le 2$, for $3 \le i \le 5$, and $\chi_2(V_i) \le 3$, for $i \ge 6$. If $G$ is maximal planar, then $\chi_2(V_3) =1$ also holds.
\item[(2)] If $G$ is an outerplanar graph, then $\chi_2(V_1) = \chi_2(V_2) = 1$ and $\chi_2(V_i) \le 2$, for $i \ge 3$.
\end{enumerate}
\end{lem}

\begin{proof}
From Lov\'asz's bound above, we derive $\chi_2(V_1) = \chi_2(V_2) = 1$ and $\chi_2(V_i) \le 2$, for $3 \le i \le 5$. By the above cited results, we have $\chi_2(V_i) \le \chi_2(G) \le 3$ for all $i \ge 6$ when $G$ is planar, and $\chi_2(V_i) \le \chi_2(G) \le 2$ for $i \ge 3$ when $G$ is outerplanar. If $G$ is maximal planar, then from Lemma \ref{lem_max-p-op-d} we obtain $\chi_2(V_3) \le \chi(V_3) = 1$ and thus $\chi_2(V_3) =1$.
\end{proof}

The previous lemma allows us to compute some lower bounds on $\tworeg(G)$ for planar and outerplanar graphs $G$, while for the case $k=1$, with the current ideas and techniques alone, we cannot do better than the bounds which were already obtained for $\reg(G)$. Table \ref{table-p-op-2} collects the bounds on $\tworeg(G)$ we have computed for planar graphs and outerplanar graphs.
 
  \begin{figure}[!h]
    \begin{center}
  \begin{tabular}{|c|c|c|c|}
        \hline
                &   Bench Mark   & Bound obtained using    &    Bound provided\\[-1ex]
                & from Prop. \ref{prop1} & Lemma \ref{lem-p-o-2} & Edge Maximality  \\
                \hline
                \hline
         \multicolumn{4}{|c|} {{\bf Planar Graphs}}\\ 
         \hline
        $\delta = 1$ & $\frac{1}{33}n$ & $\frac{4}{83}(n+3)$ & -- \\ \hline
        $\delta = 2$ & $\frac{1}{27}n$ & $\frac{3}{55}(n+4)$ & -- \\ \hline
        $\delta = 3$ & $\frac{1}{21}n$ & $\frac{1}{16}(n+4)$ &  $\frac{1}{14}(n+6)$ \\ \hline
        $\delta = 4$ & $\frac{1}{15}n$ & $\frac{2}{23}(n+6)$ & $\frac{2}{23}(n+6)$\\ \hline
        $\delta = 5$ & $\frac{1}{9}n$ & $\frac{1}{7}(n+12)$ & $\frac{1}{7}(n+12)$ \\  \hline
               \hline
        \multicolumn{4}{|c|} {{\bf Outerplanar Graphs }}\\
        \hline
        $\delta = 2$ & $\frac{1}{10}n$ & $\frac{1}{8}(n+3)$ & $\frac{1}{8}(n+3)$ \\ \hline
    \end{tabular}
    \captionof{table}{Lower bounds on $\tworeg(G)$}\label{table-p-op-2}
        \end{center}
    \end{figure}

\section{Complexity}

The well known maximum independent set decision problem (MIS) can be stated the following way.

\noindent
{\bf MIS-problem}\\
{\sc Instance:} a graph $G$ and an integer $p \ge 1$.\\
{\sc Question:} Is there an independent set of $G$ of size $\ge p$?

It is well-known that the MIS-problem is NP-complete even when it is restricted to cubic graphs \cite{Moh} or, moreover, to planar cubic graphs \cite{AlKa}. The corresponding maximum $k$-independent set decision problem and the maximum regular $k$-independent set decision problem are given below.

\noindent
{\bf kMIS-problem}\\
{\sc Instance:} a graph $G$ and an integer $p \ge 1$.\\
{\sc Question:} Is there a $k$-independent set of $G$ of size $\ge p$?

\noindent
{\bf reg-kMIS-problem}\\
{\sc Instance:} a graph $G$ and an integer $p \ge 1$.\\
{\sc Question:} Is there a regular $k$-independent set of $G$ of size $\ge p$?\\

We will show that the kMIS-problem as well as the reg-kMIS problem stated for certain regular graphs are both NP-complete, too. We will show this giving a reduction from the MIS-problem to the maximum $k$-independent set decision problem (kMIS). We fomulate both decision problems the following way.

\begin{prop} \mbox{}\\[-4ex]
\begin{enumerate}
\item[(1)] The kMIS-problem is NP-complete.
\item[(2)] The reg-kMIS problem is NP-complete for $d$-regular graphs, where $d = ((k+1)r+k)$, $r \ge 3$.
\end{enumerate}
\end{prop}

\begin{proof}
(1) Let $(G,p)$ be an instance of the MIS-problem. We will reduce this problem to an instance of the kMIS-problem. To this end, we construct a graph $H$ the following way. Replace each vertex $v \in V(G)$ by a copy of $K_{k+1}$ and denote this copy by $K_v$. Join all vertices from $K_u$ and $K_v$ by an edge if $u$ and $v$ are adjacent in $G$. We denote the obtained graph $H$ by $G_{k+1}$. Now we prove the following claims.\\
{\it Claim 1: $\alpha(H) = \alpha(G)$.}\\
Let $S$ be a maximum independent set in $G$. For each $v \in V(G)$, select a  vertex $v^* \in V(K_v)$. Then the set $S^* = \{ v^* \;:\; v \in S\}$ is an independent set in $H$ and thus $\alpha(H) \ge \alpha(G)$. Let now $U$ be a maximum independent set in $H$. Then $U$ contains at most one vertex from each $K_v$ and thus $U^* = \{u \;:\; U \cap V(K_u) \neq \emptyset\}$ is an independent set in $G$. Hence, we have $\alpha(G) \ge \alpha(H)$ and we are done. \\
{\it Claim 2: $\alpha_k(H) = (k+1) \alpha(G)$.}\\
Let $S$ be a maximum $k$-independent set of $H$. Since $H[S]$ has maximum degree at most $k$, $\chi(H[S]) \le k+1$ and thus $S$ can be split into $k+1$ independent sets. Hence, $\alpha_k(H) \le (k+1) \alpha(H)$ and, by Claim 1, $\alpha_k(H) \le (k+1)\alpha(G)$. On the other hand, let $S$ be a maximum independent set of $G$ and let $S^* = \cup_{v \in S} V(K_v)$. Then $S^*$ is a $k$-independent set of $H$ which yields $\alpha_k(H) \ge (k+1) \alpha(G)$.\\
By Claim 2, solving the MIS-problem for $(G, p)$ can be reduced to an instance $(H, (k+1)p)$ of the kMIS-problem.\\
(2) As it is mentioned above, the MIS-problem is NP-complete for $r$-regular graphs already for $r \ge 3$. Because of the above reduction and since $G_{k+1}$ is $((k+1)r+k)$-regular provided $G$ is $r$-regular, $r \ge 3$, the reg-kMIS-problem is NP-complete for $((k+1)r+k)$-regular graphs.
\end{proof}

However, when we restrict our attention to a hereditary family of graphs (i.e. closed under induced subgraphs), the situation may look different.

\begin{prop}\label{thm-hereditary}
Let $\mathcal{F}$ be a hereditary family of graphs for which finding the $k$-independence number $\alpha_k(G)$ is solvable in polynomial time, say $p(n)$, for each $G \in \mathcal{F}$ of order $n$. Then finding the regular $k$-independence number $\kreg(G)$ is also solvable in polynomial time $\mathcal{O}(np(n))$ for each $G \in \mathcal{F}$ of order $n$.
\end{prop}

\begin{proof}
Arrange the adjacency matrix in non-decreasing order with respect to the degrees of the vertices (both in rows and columns). Then the induced subgraph $G_ j$ of the vertices of degree $j$ is given by the principal submatrix of the corresponding rows and columns of the vertices of degree $j$. Now, because of the hereditary property, $\alpha_k(G_ j)$ can be computed  in time $p(n_j) \le p(n)$, where $n_j = n(G_j)$ and $\sum_{j=0}^{n-1} n_j = n$.  So, taking into account the construction of $G_j$, we need, rather crudely, $\mathcal{O}(n_j^2 + p(n))$-time to compute $\kreg(G_j)$. Summing over all $j= 1, \ldots, n$, and using convexity, we get $\sum_{j=0}^{n-1} \mathcal{O}(n_j^2 + p(n)) =  \mathcal{O}(n^2) + np(n) = \mathcal{O}(np(n))$. 
\end{proof}

In view of Proposition \ref{thm-hereditary}, since the independent set problem ($k=0$) for claw-free graphs, perfect graphs, bounded tree-width graphs and boundend clique-width graphs can be solved in polynomial-time \cite{CoMaRo, Gro, GrCl, Sbi}, the regular independence number can be computed in polynomial time in all these cases as well. The same occurs with the maximal outerplanar graphs, for which the independence problem can be solved in linear time \cite{CoRo, EsGuWa}.

\section{Open problems}

We close this paper with the following open problems.


\begin{probl}
Let $G$ be a $k$-tree of order $n \ge k+t+2$, where $t \ge 0$ is an integer. Is the bound $\chi(V_{k+t}(G)) \le \frac{1}{2}(t^2+t+2)$ optimal or can it be improved?
\end{probl}



\begin{probl}
Improve upon the bounds on $\reg(G)$ given in Section 3.
\end{probl}

\begin{probl}
Improve upon the benchmark bounds on $\jreg(G)$ for $k$-trees when $j > 0$ and $k > 1$.
\end{probl}

\begin{probl}
Prove (or disprove) that the maximum regular $k$-independent set decision problem is NP-complete for all $r$-regular graphs, $r \ge \max\{3, k+1\}$.
\end{probl}



\begin{thebibliography}{100}

\bibitem{AlBe} M. O. Albertson, D. M. Berman, A conjecture on planar graphs. In: Bondy, J.A., Murty, U.S.R. (eds),  \emph{Graph Theory and Related Topics}, Academic Press, 357, 1979.
\bibitem{AlBo}  M. O. Albertson, D. L. Boutin, Lower bounds for constant degree independent sets, Graph theory and applications (Hakone, 1990), Discrete Math. {\bf 127} (1994), no. 1--3, 15--21.
\bibitem{AlKa} P. Alimonti, V. Kann, Hardness of approximating problems on cubic graphs, LNCS {\bf 1203}, Springer, Heidelberg (1997), 288--298.
\bibitem{ApHa} K. Appel, W. Haken, Every planar map is four colorable I: discharging, Illinois J. Math. {\bf 21} (1977),  429--490.
\bibitem{ApHaKo} K. Appel, W. Haken, J. Koch, Every planar map is four colorable II: reducibility, Illinois J. Math. {\bf 21} (1977),  491--567.
\bibitem{AlMuTh} N. Alon, D. Mubayi, R. Thomas, Large induced forests in sparse graphs, J. Graph Theory {\bf 38} (2001), no. 3, 113--123. 
\bibitem{BieWil} T. Biedl, D. F. Wilkinson, Bounded-degree independent sets in planar graphs, Theory Comput. Syst. {\bf 38} (2005), no. 3, 253--278.
\bibitem{BoDuWo} P. Bose, V. Dujmovi\'c, D. R. Wood, Induced subgraphs of bounded degree and bounded treewidth, Graph-theoretic concepts in computer science, 175--186, Lecture Notes in Comput. Sci., 3787, Springer, Berlin, 2005. 
\bibitem{Br} R. L. Brooks, On colouring the nodes of a network. Proc. Cambridge Philos. Soc. {\bf 37} (1941), 194--197.
\bibitem{CaHa} Y. Caro, A. Hansberg, New approach to the $k$-independence number of a graph, Electron. J. Combin. {\bf 20} (2013), no. 1, Paper 33, 17 pp.
\bibitem{CHH} Y. Caro, A. Hansberg, M. Henning, Fair domination in graphs, Disc. Math. {\bf 312} (2012), 2905--2914.
\bibitem{CaWe} Y. Caro, D. B. West, Repetition number of graphs, Electron. J. Combin. {\bf 16} (2009), no. 1, Research Paper 7, 14 pp.
Discrete Math. {\bf 127} (1994), no. 1--3, 15--21.
\bibitem{CoRo} D.G. Corneil, U. Rotics, On the relationship between clique-width and treewidth, SIAM J. Comput. {\bf 34} (2005), 825--847.
\bibitem{CoMaRo} B. Courcelle, J. A. Makowsky, U. Rotics, Linear time solvable optimization problems on graphs of bounded clique-width, Theory Comput. Syst. {\bf 33} (2) (2000), 125--150.
\bibitem{CoCoWo} L. J. Cowen, R. H. Cowen, D. R. Woodall, Defective colorings of graphs in surfaces: partitions into subgraphs of bounded valency, J. Graph Theory {\bf 10} (1986), no. 2, 187--195.
\bibitem{CoGoJe} L. Cowen, W. Goddard, C. E. Jesurum, Defective coloring revisited, J. Graph Theory {\bf 24} (1997), no. 3, 205--219. 
Discrete Math. {\bf 127} (1994), no. 1--3, 15--21.
\bibitem{EsGuWa} W. Espelage, F. Gurski, E. Wanke, Deciding clique-width for graphs of bounded tree-width, J. Graph Algorithms Appl. {\bf 7} (2003) 141--180.
\bibitem{GaJo} M. R. Garey, D. S. Johnson, \emph{Computers and intractability. A guide to the theory of NP-completeness.} A Series of Books in the Mathematical Sciences. W. H. Freeman and Co., San Francisco, Calif., 1979. x+338 pp.
\bibitem{GrCl} http://www.graphclasses.org/classes/gc\underline{ }470.html
\bibitem{Fri} M. Frick, A survey of $(m, k)$-colorings, in J. Gimbel, J. W. Kennedy, and L. V. Quintas (eds.). In: \emph{Quo Vadis, Graph Theory?}, vol. 55 of Annals of Discrete Mathematics, pages 45--58. Elsevier Science Publishers, New York (1993).
Discrete Math. {\bf 127} (1994), no. 1--3, 15--21.
\bibitem{Gro} M. Gr\"otschel, L. Lov\'asz, A. Schrijver, Geometric algorithms and combinatorial optimization, Springer, Berlin (1988).
\bibitem{HaPe} A. Hansberg, R. Pepper, On $k$-domination and $j$-independence in graphs, Discrete Appl. Math. 161 (2013), no. 10--11, 1472--1480.
\bibitem{LiWhi} D. R. Lick, A. T. White, $k$-degenerate graphs, Canad. J. Math. {\bf 22} (1970), 1082--1096. 
\bibitem{Lov} L. Lov\'asz. On decompositions of graphs, Studia Sci. Math Hungar. {\bf 1} (1966) 237--238.
Discrete Math. {\bf 127} (1994), no. 1--3, 15--21.
\bibitem{Moh} B. Mohar, Face covers and the genus problem for apex graphs, J. Combin. Theory (B), {\bf 82}(1) (2000), 102--117.
\bibitem{Rack} T. Rackham, The number of defective colorings of graphs on surfaces, J. Graph Theory {\bf 68} (2011), no. 2, 129--136. 
Discrete Math. {\bf 127} (1994), no. 1--3, 15--21.
\bibitem{Sbi} N. Sbihi, Algorithme de recherche d'un stable de cardinalit\'e maximum dans un graphe sans \'etoile, Discrete Math. {\bf 29} (1980), 53--76.
\end{thebibliography}
\end{document}